\documentclass[12pt]{amsart} 
\usepackage[margin=0.8in]{geometry}
\usepackage{amsmath,amsfonts,amssymb,amsmath,amscd}
\usepackage[active]{srcltx}
\usepackage{graphicx,array}
\usepackage{amsthm}
\usepackage{amsbsy}
\usepackage{tikz}
\usepackage{soul}

\usetikzlibrary{graphs}
\usepackage{tkz-berge}
\usepackage{float}
\usepackage{enumerate}

\usepackage[latin1]{inputenc}

\numberwithin{equation}{section}

\newtheorem{theorem}{Theorem}[section]

\newtheorem{proposition}[theorem]{Proposition}
\newtheorem{corollary}[theorem]{Corollary}

\newtheorem{example}[theorem]{Example}
\newtheorem{remark}[theorem]{Remark}

\newcommand{\END}{\hfill \mbox{\raggedright $\Diamond$}}

\newcommand{\R}{\mathbf{R}}

\def \V{\mbox{${\mathcal{V}}$}}

\def \V{\mbox{$V$}}

\def \r{\mbox{${\bf {R}}$}}

\title{Quotients and lifts of symmetric directed graphs}
\author{Manuela Aguiar}
\address{Manuela Aguiar, Faculdade de Economia, Centro de Matem\'atica, Universidade do Porto,
Rua Dr Roberto Frias, 4200-464 Porto, Portugal.}
\email{maguiar@fep.up.pt}
\author{Ana Dias}
\address{Ana Dias, Departamento de Matem\'atica, Centro de Matem\'atica, Universidade do Porto,
Rua do Campo Alegre, 687, 4169-007 Porto, Portugal}
\email{apdias@fc.up.pt}
\author{Miriam Manoel}
\address{Miriam Manoel, Departmento de Matem\'atica, ICMC-Universidade de S\~ao Paulo,  C.P. 668, 13560-970 S\~ao Carlos SP, Brazil}
\email{miriam@icmc.usp.br}

\date{\today}

\begin{document}

\maketitle 

\begin{abstract}
 Given a directed graph, an equivalence relation on the graph vertex set  is said to be balanced if, for every two vertices  in the same equivalence class, the number of directed edges from vertices of each equivalence class directed to each of the two vertices is the same. In this paper we describe the quotient and lift graphs of symmetric directed graphs associated with balanced equivalence relations on the associated vertex sets. In particular, we characterize the quotients and lifts which are also symmetric.  We end with an application of these results to gradient and Hamiltonian coupled cell systems, in the context of the coupled cell network formalism of Golubitsky, Stewart and T\"or\"ok (Patterns of synchrony in coupled cell networks with multiple arrows. {\em SIAM Journal of  Applied Dynamical  Systems} {\bf 4} (1) (2005) 78--100).
\end{abstract}

\begin{flushleft}
 
\vspace*{2mm}

%05C50  	Graphs and linear algebra (matrices, eigenvalues, etc.)
%05C15  	Coloring of graphs and hypergraphs
%05C90  	Applications 
%05C82  	Small world graphs, complex networks
\end{flushleft}

\section{Introduction}
Let $G$ be a directed graph with a finite set of vertices $V =\{ 1, \ldots, n\}$ and a finite set of edges $E$. An equivalence relation $\bowtie$ on $V$ is {\em balanced} when, for any  two vertices $u$ and $v$ in the same $\bowtie$-class, the number of directed edges in $E$ from vertices of each $\bowtie$-class directed to $u$ and to $v$ is the same.  Coloring the vertices of the graph by choosing one color for each  $\bowtie$-class,  the  {\em coloring} is said to be {\em balanced} when $\bowtie$ is balanced. 

We recall that the {\em adjacency matrix} of $G$ is the $n \times n$ matrix, denoted by $A_{G}$, with non-negative integer entries, where the $(i,j)$ entry is the number of directed edges in $E$ from vertex $j$ to vertex $i$. Also, given an equivalence relation $\bowtie$ on $V$,  we define the {\em polydiagonal subspace} of $\R^n$ associated with $\bowtie$ by
\[
\Delta_{\bowtie} = \left\{ {\bf x} \in \R^n:\ x_u = x_v \mbox{ whenever } 
u \bowtie v,\, \forall u,v \in V \right\}\, .
\]
Trivially, we have that: 

\begin{theorem} \label{thm:bal_equi_inv}
An equivalence relation $\bowtie$ on the set of vertices of a directed graph $G$ is balanced if and only if the graph adjacency matrix $A_G$ leaves the polydiagonal subspace $\Delta_{\bowtie}$ invariant. 
\end{theorem}

In \cite{S07} it is proved that the set of all balanced equivalence relations of a graph forms a complete lattice taking the relation of refinement. Moreover, an algebraic and algorithmic description of this lattice using the graph adjacency matrix eigenvector structure is given in \cite{AD14}.

Given a balanced equivalence relation on the vertex set $V$ of a directed graph $G$, with edge set $E$, the graph with vertex set $V/\bowtie$ and edge set given by the edges $([u]_{\bowtie},[v]_{\bowtie})$, where $(u,v) \in E$, is the {\em quotient graph} of $G$ by $\bowtie$ and is denoted by $Q = G/\bowtie$. Thus, the set of vertices of $Q$ is formed by one vertex corresponding to each $\bowtie$-equivalence class; the edges are the projections of the edges  in the original graph.  
A graph $G$ with a finite set of vertices $V$  and a finite set of edges $E$ is a {\em lift} of a graph $Q$ when $G = Q/\bowtie$ for some balanced relation $\bowtie$ on $V$. 

The {\it valency} or the {\it indegree} of a vertex is the number of arrows directed to it.  A  graph is {\it regular} when all the vertices have the same valency. It follows, trivially, that: 

\begin{remark} \normalfont
Given a balanced equivalence relation $\bowtie$ on the vertex set of a directed graph $G$, the valency of the vertices is preserved under the quotient and lifting operations.  Moreover, $G$ is regular if and only if $G/\bowtie$ is regular.
\END
\end{remark}

A relevant observation concerning lifts and quotient graphs is the following:

\begin{remark} \normalfont
Fixing $G$ and a balanced equivalence relation $\bowtie$ of $G$, the quotient graph $Q = G / \bowtie$ is unique. Varying the balanced equivalence relation $\bowtie$, in general, the quotient graph $Q = G / \bowtie$ also varies. Fixing $Q$, say with $k$ vertices and choosing $n >k$, in general, there can be several different lifts of $Q$ with $n$ vertices. See~\cite{ADGL09} for details and a systematic method of enumerating the lifts of a general (quotient) graph $Q$.
\END
\end{remark}

A directed graph $G$ is {\em symmetric} when, for every edge in $E$, the corresponding inversed edge also belongs to $E$.  Using the graph adjacency matrix $A_G$, the graph $G$ is symmetric if and only if $A_G = A_G^t$ where $A_G^t$ denotes the transpose matrix of $A_G$.

This work focuses on symmetric directed graphs, sometimes also called  {\em bidirected graphs}, and the associated quotient and lift graphs, using balanced equivalence relations. As it can easily be seen, in general, the property of a graph being bidirected may not be preserved under the lifting and quotient operations. In fact, for most of the symmetric directed graphs, there are quotient and lift graphs that are not symmetric. We illustrate that with a few examples. 

\begin{example} \normalfont Consider the symmetric Petersen graph $G$ pictured in Figure~\ref{fig:Peter}. Let $\bowtie$ be the equivalence relation on the set of vertices of $G$  with classes 
\[
\bowtie = \left\{ \{ 1,2,3,4,5\},\, \{6,7,8,9,10\}\right\},
\]
 which correspond to the coloring of the vertices of the Petersen graph presented in Figure~\ref{fig:colorPeter}. We have that $\bowtie$ is balanced. Graphically, we see that as every grey vertex of $G$ receives two edges from grey vertices and one edge from one black vertex. Similarly, every black vertex of $G$ is the input of two edges from black vertices and one edge from one grey vertex. Equivalently, we have that the order-10 adjacency matrix of $G$ leaves invariant the linear space 
 $$\Delta_{\bowtie} = \{ x: \, x_{1} = \cdots = x_{5},\  x_{6} = \cdots = x_{10}\}\, .$$ 
Note that the quotient graph $G/ \bowtie$ is the $2$-vertex symmetric graph $Q$ in Figure~\ref{fig:simple}, which has adjacency matrix 
\[ A_{Q}= \left( 
\begin{array}{cc}
2 & 1 \\
1 & 2 
\end{array} \right)\, .
\]
Moreover, the restriction of $A_G$ to $\Delta_{\bowtie}$ is similar to $A_Q$.

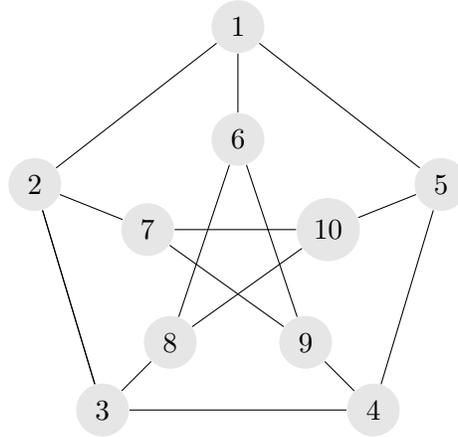
\begin{figure}[H] 
\begin{center}
\begin{tikzpicture}
  [scale=.3,auto=left,every node/.style={circle,fill=black!10}]
  \node (n1) at (9,4) {\small{1}};
  \node (n6) at (9,-1)  {\small{6}};
  \node  (n7) at (5,-5)  {\small{7}};
  \node  (n10) at (13,-5) {\small{10}};
 \node  (n2) at (0,-3)  {\small{2}};
 \node (n5) at (18,-3) {\small{5}};
\node (n8) at (6,-10) {\small{8}};
\node (n9) at (12,-10) {\small{9}};
\node (n3) at (3,-13) {\small{3}};
\node (n4) at (15,-13) {\small{4}};

 \foreach \from/\to in {n1/n2,n1/n5,n1/n6,n2/n7,n2/n3, n3/n2, n3/n4, n3/n8,n4/n5,n4/n9,n5/n10,n7/n9,n7/n10,n8/n10,n6/n8,n6/n9}
  \draw (\from) -- (\to); 
\end{tikzpicture}
\end{center}
\caption{The Petersen graph is an example of a $10$-vertex symmetric directed graph, where all the edges are bidirectional.}
\label{fig:Peter}
\end{figure}

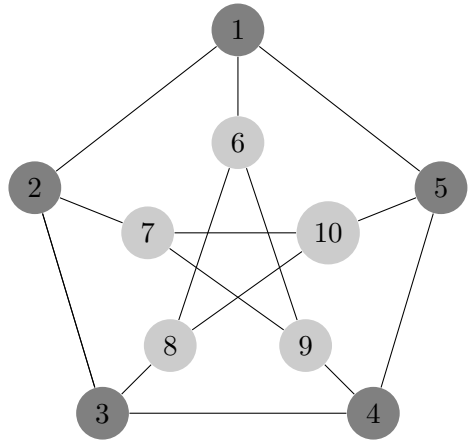
\begin{figure}[H] 
\begin{center}
\begin{tikzpicture}
  [scale=.3,auto=left,every node/.style={circle,fill=black!10}]
  \node[fill=black!50] (n1) at (9,4) {\small{1}};
  \node[fill=black!20]  (n6) at (9,-1)  {\small{6}};
  \node[fill=black!20]  (n7) at (5,-5)  {\small{7}};
  \node[fill=black!20]  (n10) at (13,-5) {\small{10}};
 \node[fill=black!50]  (n2) at (0,-3)  {\small{2}};
 \node[fill=black!50]  (n5) at (18,-3) {\small{5}};
\node[fill=black!20] (n8) at (6,-10) {\small{8}};
\node[fill=black!20] (n9) at (12,-10) {\small{9}};
\node[fill=black!50] (n3) at (3,-13) {\small{3}};
\node[fill=black!50] (n4) at (15,-13) {\small{4}};

 \foreach \from/\to in {n1/n2,n1/n5,n1/n6,n2/n7,n2/n3, n3/n2, n3/n4, n3/n8,n4/n5,n4/n9,n5/n10,n7/n9,n7/n10,n8/n10,n6/n8,n6/n9}
  \draw (\from) -- (\to); 
\end{tikzpicture}
\end{center}
\caption{A $2$-color balanced coloring $\bowtie$ of the Petersen graph $G$: every grey vertex receives two edges from grey vertices and one edge from one black vertex; every black vertex is the input of two edges from black vertices and one edge from one grey vertex.  The quotient $G/\bowtie$ is the $2$-vertex graph of Figure~\ref{fig:simple}. }
\label{fig:colorPeter}
\end{figure}

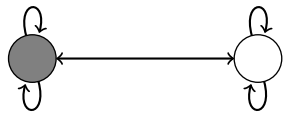
\begin{figure}[H] 
\begin{center}
\begin{tikzpicture}
  \node  (n1) at (5,2)  {\small{1}};
  \node  (n2) at (8,2)  {\small{2}};
  \AddVertexColor{black!50}{n1}
  \AddVertexColor{white}{n2}
\draw[->, thick] (n1) edge[loop above] (n1); 
\draw[<-, thick] (n1) edge[loop below ] (n1); 
\draw[->, thick] (n2) edge[loop below] (n2); 
\draw[->, thick] (n2) edge[loop above] (n2); 
 \draw[<->, thick] (n1) edge[bend left=0] (n2); 
\end{tikzpicture} 
\end{center}
\caption{A $2$-vertex symmetric directed graph $Q$. There are symmetric and non symmetric graphs $G$ with balanced colorings $\bowtie$ such that $Q = G/\bowtie$. Equivalently, there are symmetric and non-symmetric lifts $G$ of $Q$.}
\label{fig:simple}
\end{figure}

Considering now the equivalence relation $\bowtie_2$ on the set of vertices of $G$  with classes 
$$\bowtie_2 = \left\{ \{1,3,9,10\},\, \{ 2,7\},\, \{4,5, 6, 8\}\right\},$$
 corresponding to the coloring of the vertices of the Petersen graph presented in Figure~\ref{fig:3colorPeter}, we have that $\bowtie_2$ is also balanced and 
 $Q_2 = G/ \bowtie_2$ is the non-symmetric graph of Figure~\ref{fig:3colorquo}. The quotient graph $Q_2$ has the non-symmetric adjacency matrix 
 $$
 A_{Q_2} = \left( 
\begin{array}{ccc}
0 & 1 & 2 \\
2 & 1 & 0 \\
2 & 0 & 1 
\end{array} \right)\, .
 $$
\END
\end{example}

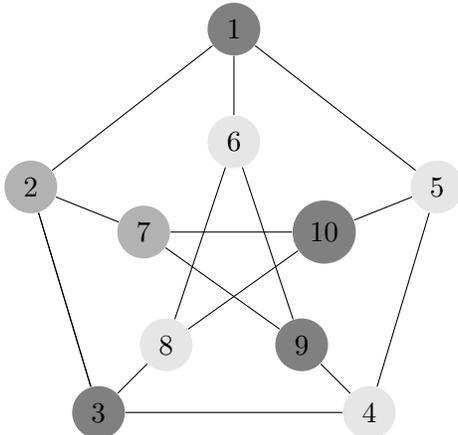
\begin{figure}[H] 
\begin{center}
\begin{tikzpicture}
  [scale=.3,auto=left,every node/.style={circle,fill=black!10}]
  \node[fill=black!50] (n1) at (9,4) {\small{1}};
  \node[fill=black!10]  (n6) at (9,-1)  {\small{6}};
  \node[fill=black!30]  (n7) at (5,-5)  {\small{7}};
  \node[fill=black!50]  (n10) at (13,-5) {\small{10}};
 \node[fill=black!30]  (n2) at (0,-3)  {\small{2}};
 \node[fill=black!10]  (n5) at (18,-3) {\small{5}};
\node[fill=black!10] (n8) at (6,-10) {\small{8}};
\node[fill=black!50] (n9) at (12,-10) {\small{9}};
\node[fill=black!50] (n3) at (3,-13) {\small{3}};
\node[fill=black!10] (n4) at (15,-13) {\small{4}};

 \foreach \from/\to in {n1/n2,n1/n5,n1/n6,n2/n7,n2/n3, n3/n2, n3/n4, n3/n8,n4/n5,n4/n9,n5/n10,n7/n9,n7/n10,n8/n10,n6/n8,n6/n9}
  \draw (\from) -- (\to); 
\end{tikzpicture}
\end{center}
\caption{A $3$-color balanced coloring $\bowtie_2$ of the Petersen graph $G$. The quotient $G/\bowtie_2$ is the $3$-vertex graph of Figure~\ref{fig:3colorquo}. }
\label{fig:3colorPeter}
\end{figure}

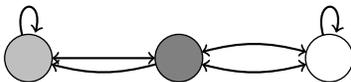
\begin{figure}[H] 
\begin{center}
\begin{tikzpicture}
  \node  (n1) at (5,2)  {\small{1}};
  \node  (n2) at (7,2)  {\small{2}};
   \node  (n3) at (9,2)  {\small{3}};
  \AddVertexColor{black!50}{n2}
  \AddVertexColor{black!25}{n1}
  \AddVertexColor{white}{n3}
\draw[->, thick] (n1) edge[loop above] (n1); 
\draw[->, thick] (n3) edge[loop above] (n3); 
 \draw[<->, thick] (n1) edge[bend left=0] (n2); 
 \draw[->, thick] (n2) edge[bend left=15] (n1);
  \draw[<->, thick] (n2) edge[bend left=-15] (n3); 
  \draw[<->, thick] (n2) edge[bend left=15] (n3); 
\end{tikzpicture} 
\end{center}
\caption{A $3$-vertex non-symmetric directed graph $Q_2$, which is the quotient  $G/\bowtie_2$ where $G$ is  the (symmetric) Petersen graph and  $\bowtie_2$ is the balanced coloring  presented in  Figure~\ref{fig:3colorPeter}.}
\label{fig:3colorquo}
\end{figure}

In  Section~\ref{sec:QLSG} we describe the quotient and lift graphs of symmetric directed graphs associated with balanced equivalence relations on their set of vertices. In particular, we characterize the quotient and lift graphs of symmetric directed graphs which are also symmetric directed graphs. In Section~\ref{sec:ApCCN}, these results are applied to gradient and Hamiltonian coupled cell systems, in the context of the formalism of Golubitsky, Stewart and co-workers~\cite{SGP03,GST05, GS06}, where coupled cell systems are dynamical systems associated with graphs (coupled cell networks). Networks of gradient and Hamiltonian coupled systems have received some attention lately, see for example  \cite{Bronski et al} and \cite{Rink Sanders} for gradient networks and  \cite{BCP15} and \cite{T17} for Hamiltonian networks, as well as  references therein.

\section{Quotients and lifts of symmetric graphs} \label{sec:QLSG}

In this section, we characterize quotients and lifts of symmetric directed graphs. In particular, we identify the quotients and the lifts of symmetric (bidirected) graphs that are also symmetric (bidirected).

\begin{theorem}\label{thm:quobidi}
An $m$-vertex graph $Q$ is a quotient graph of an $n$-vertex symmetric directed graph $G$ by a balanced equivalence relation on the set of vertices of $G$, with $m < n$, if and only if the entries of the adjacency matrix $A_{Q} = [q_{ij}]_{m \times m}$ of $Q$ satisfy the following: there are positive integers $k_1, \ldots, k_m$ summing $n$ and such that $k_i q_{ij} = k_j q_{ji}$. 
\end{theorem}
\begin{proof} Let $G$ be an $n$-vertex symmetric directed graph. Thus its adjacency matrix is an $n\times n$ symmetric matrix. Consider a balanced equivalence relation on its set of vertices  and the corresponding quotient graph $Q$ with adjacency matrix $A_{Q} = [q_{i,j}]_{m \times m}$.  
Let $i= \{ i_1, \ldots, i_{k_i} \}$ and $j= \{ j_1, \ldots, j_{k_j} \}$ be any two such  equivalence classes.  Since the relation is balanced, it follows that the number of edges directed from vertices of the class $j$ to each vertex of the class $i$ is the constant $q_{ij}$. Thus, the total number of directed edges from vertices of $j$ to vertices of $i$ is then $k_i q_{ij}$. Moreover, as $G$ is bidirectional, all these directed edges are bidirectional. It then follows that the total number of directed edges from vertices of class $i$ to vertices of class $j$, which is $k_j q_{ji}$, satisfy 
$$k_i q_{ij} = k_j q_{ji}\, .$$

\noindent Conversely,  let  $Q$ be a graph with adjacency matrix $A_{Q} = [q_{ij}]_{m \times m}$ such that   there are positive integers $k_1, \ldots, k_m$ summing $n>m$ and such that $k_i q_{ij} = k_j q_{ji}$ for all $i,j=1, \ldots, m$.  We construct a symmetric matrix $A$ of order $n$ from the matrix $A_{Q}$ in the following way: \\
(i) Consider $A$ as a block matrix with $m \times m$ blocks $Q_{ij}$, with $1 \le i,j \le m$. \\
(ii) For $i$ and $j$ such that $i \leq j$, define the submatrix $Q_{ij}$ of $A$ as a matrix of order $k_i \times k_j$ with the row sum $q_{ij}$ and column sum   $q_{ji}$. \\
(iii) For $j <  i$, set the submatrix $Q_{ji} = Q_{ij}^T$. \\
It follows that $A$ is symmetric. Let $G$ be the $n$-vertex graph with set of vertices $\{1, \ldots, n\}$ and adjacency matrix $A$. It follows that $G$ is bidirectional. Moreover, by construction, the equivalence relation $\bowtie$ on the set of vertices of $G$  with classes $I_1 = \{1, \ldots, k_1\}, \ I_2 = \{ k_1 + 1, \ldots, k_1 + k_2\},\ \ldots, I_m = \{ k_{m-1} + 1, \ldots, k_m\}$, is balanced and $Q$ is the quotient of $G$ by $\bowtie$. That is, the graph $G$ is a bidirectional lift of the graph $Q$.
\end{proof}

\begin{remark}\normalfont 
In Lemma 4.3 of \cite{DP09}, it is obtained that in order for a graph to be a quotient of a regular symmetric graph of valency $v$ then its adjacency matrix satisfies the conditions stated in  Theorem~\ref{thm:quobidi} and the vector $(k_1, \ldots, k_m)$ is a left eigenvalue of that matrix associated with the eigenvalue $v$. 
\END  
\end{remark}

\begin{remark}\normalfont 
We recall that a square matrix $(q_{ij})$  is a {\em combinatorially symmetric matrix} when the entries satisfy the condition, $q_{ij} = 0$ if and only if  $q_{ji} = 0$, for all $i,j$ . From Theorem~\ref{thm:quobidi}, it follows that the adjacency matrix of a quotient graph of a symmetric directed graph by a balanced equivalence relation on the vertex set  is a combinatorially symmetric matrix. Thus  a necessary condition for $Q$ to be a quotient graph $G/\bowtie$ of a symmetric directed  graph $G$ under a balanced equivalence relation $\bowtie$  is that, for any two vertices in $Q$, if they are connected then there is at least one bidirectional connection between them. \END 
\end{remark}

\begin{remark}\normalfont \label{rmk:quobidi}
Let  $Q$ be an $m$-vertex graph under the hypothesis of Theorem~\ref{thm:quobidi}, that is, taking $A_{Q} = [q_{ij}]_{m \times m}$ to be its adjacency matrix, then there are positive integers $k_1, \ldots, k_m$ summing $n>m$ and such that $k_i q_{ij} = k_j q_{ji}$. We have that:\\
(i) If $k_1, \ldots, k_m$ are coprime,  then the possible symmetric directed networks $G$ constructed following the proof of the theorem are symmetric directed lifts of $Q$ with the minimum number of vertices, $n = \sum_{i=1}^m k_i$. \\
(ii)  Trivially, in general, there are lifts $G$ of the graph $Q$, where $Q= G/ \bowtie$ for balanced $\bowtie$, that are not symmetric directed graphs.
\END 
\end{remark}

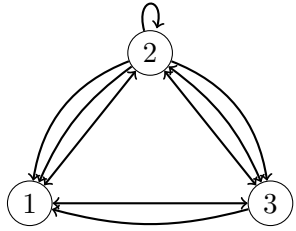
\begin{figure}[H] 
\begin{center}
\tiny{ 
\begin{tikzpicture}
  [scale=.4,auto=left,node distance=1.5cm, every node/.style={circle,draw}]  
  \node  (n1) at (9,2) [fill=white] {\small{2}} edge [loop above, thick] ();
  \node  (n2) at (5,-3) [fill=white] {\small{1}};
  \node  (n3) at (13,-3) [fill=white]{\small{3}};
 \foreach \from/\to in {n1/n2, n2/n3, n3/n1}  \draw[<->, thick]  (\from) -- (\to); 
 \draw[->, thick] (n1) edge[bend left=-15] (n2); \draw[->, thick] (n1) edge[bend left=-30] (n2); \draw[->, thick] (n3) edge[bend left=15] (n2);
 \draw[<->, thick] (n1) edge[bend left=15] (n3);
\draw[->, thick] (n1) edge[bend left=30] (n3); 
\end{tikzpicture} 
}
\end{center}
\caption{A $3$-vertex directed graph $Q$ which is not symmetric and admitting lifts $G$ which are symmetric graphs and such that $Q = G/ \bowtie$ for balanced $\bowtie$.}
\label{fig:quotbidi}
\end{figure}

\begin{example} \label{ex:quobidi}
\normalfont Consider the directed graph $Q$ given in Figure~\ref{fig:quotbidi}. The adjacency matrix  of $Q$ is
\[ A_{Q} = \left( 
\begin{array}{ccc}
0 & 3 & 2 \\
1 & 1 & 2 \\
1 & 3 & 0\\
\end{array} \right).
\]
The entries of $A_{Q}$ satisfy $q_{12} = 3 q_{21}$, $q_{13} = 2 q_{31}$ and $3 q_{23} = 2 q_{32}$, that is, the matrix  satisfies the hypothesis in Theorem~\ref{thm:quobidi} for $k_1 =1$, $k_2 = 3$ and $k_3 = 2$. Hence, by Theorem~\ref{thm:quobidi}  the graph $Q$ is a quotient of a symmetric graph $G$ by some balanced equivalence relation of the set of vertices of $G$. Moreover, by Remark~\ref{rmk:quobidi} the symmetric lifts with the minimum number of vertices have $6$ vertices. In fact, the two matrices below are adjacency matrices which realize such symmetric graphs: 
{\small 
\[ \left( 
\begin{array}{c|ccc|cc}
0 & 1 & 1 & 1 & 1 & 1 \\
\hline 
1 & 1 & 0 & 0 & 1 & 1 \\
1 & 0 & 1 & 0 & 1 & 1 \\
1 & 0 & 0 & 1 & 1 & 1 \\
\hline 
1 & 1 & 1 & 1 & 0 & 0 \\
1 & 1 & 1 & 1 & 0 & 0 
\end{array} \right) ,
\quad
%\mbox{ and }
\quad
 \left( 
\begin{array}{c|ccc|cc}
0 & 1 & 1 & 1 & 1 & 1 \\
\hline 
1 & 1 & 0 & 0 & 0 & 2 \\
1 & 0 & 1 & 0 & 1 & 1 \\
1 & 0 & 0 & 1 & 2 & 0 \\
\hline 
1 & 0 & 1 & 2 & 0 & 0 \\
1 & 2 & 1 & 0 & 0 & 0 
\end{array} \right).
\]
}
\END 
\end{example}

As illustrated in the introduction and in the above example, given a symmetric directed graph $G$ and a balanced equivalence relation $\bowtie$ on its set of  vertices, the quotient $Q= G/ \bowtie$ may not be a symmetric  graph. 
In the next result we give a necessary and sufficient condition for the quotient $Q$ to be a symmetric graph.

\begin{theorem}\label{thm:bidquo}
Let $Q$ be an $m$-vertex quotient graph of an $n$-vertex symmetric directed graph $G$, with $m < n$, associated with a balanced equivalence relation $\bowtie$ on the vertices of $G$. The quotient graph $Q$ is symmetric if, and only if, for each pair of connected vertices in $Q$, the corresponding $\bowtie$-classes on the set of vertices of $G$  have the same cardinality. 
\end{theorem}

\begin{proof}
Let $G$ be an $n$-vertex symmetric directed graph and $Q$ an $m$-vertex quotient graph of $G$ associated with a balanced equivalence relation $\bowtie$, with classes $I_i$ for $i=1, \ldots, m$, where $n>m$. Denote the cardinality of class $I_i$ by $k_i$, for $i=1, \ldots, m$ and so the sum of $k_1, \ldots, k_m$ is $n$. Take the adjacency matrix $A_Q = [q_{ij}]$ of $Q$ according to the first part of the proof of Theorem~\ref{thm:quobidi}. Thus the positive integers $k_1, \ldots, k_m$ satisfy $k_i q_{ij} = k_j q_{ji}$. 
In particular, we have that $q_{ij} = 0$ if and only if $q_{ji} = 0$. Take $i,j$ such that $i\not=j$ and $q_{ij} \not= 0$. If  $q_{ij} = q_{ji}$ and  $k_i q_{ij} = k_j q_{ji}$ then  $k_i = k_j$, that is, the classes $I_i$ and $I_j$ have the same cardinality. Conversely, if $k_i = k_j$ and $k_i q_{ij} = k_j q_{ji}$, then 
$q_{ij} = q_{ji}$. Thus $A_Q$ is symmetric if, and only if, for every pair of vertices $i,j$ of $Q$ which are connected, that is, such that $q_{ij} \not=0$, the classes $I_i$ and $I_j$ have the same cardinality.  
\end{proof}

\begin{example}\normalfont  \label{example:quotientbi}
Consider the $6$-vertex symmetric directed graph $G$ in Figure~\ref{fig:exf} with adjacency matrix 
\[
A_{G} = \left( 
\begin{array}{llllll}
0 & 0 & 1 & 0 & 1 & 1 \\
0 & 0 & 0 & 1 & 1 & 1 \\
1 & 0 & 0 & 0 & 1 & 1 \\
0 & 1 & 0 & 0 & 1 & 1 \\
1 & 1 & 1 & 1 & 0 & 1 \\
1 & 1 & 1 & 1 & 1 & 0
\end{array}\right).
\]
\noindent Also, consider the following two balanced equivalence relations on the set of vertices of $G$: $\bowtie_1 = \{ \{ 1,2\},  \{3,4 \}, \{5,6 \} \}$ and $\bowtie_2 = \{ \{ 1,2, 3, 4\},  \{5 \}, \{6 \} \}$. For each relation, there are connections between vertices of any class to vertices of any other class. Since the cardinality of all the $\bowtie_1$-classes is the same, it follows from Theorem~\ref{thm:bidquo} that the corresponding quotient graph $Q_1 = G / \bowtie_1$ (Figure~\ref{fig:exq1} (a)) must be  symmetric; in fact, the adjacency matrix of $Q_1$ is the symmetric matrix
\[
\left( 
\begin{array}{lll}
0 & 1 & 2  \\
1 & 0 & 2  \\
2 & 2 & 1  
\end{array}\right)\, .
\]
Since the cardinality of the $\bowtie_2$-classes is not the same, it follows from Theorem~\ref{thm:bidquo} that the quotient graph $Q_2 = G/ \bowtie_2$  (Figure~\ref{fig:exq1} (b))  is not symmetric; in fact, the adjacency matrix of $Q_2$ is the non-symmetric matrix
\[
\left( 
\begin{array}{lll}
1 & 1 & 1  \\
4 & 0 & 1  \\
4 & 1 & 0  
\end{array}\right)\, .
\]
\END
\end{example}

\begin{figure}[H] 
\begin{center}
\tiny{ 
\begin{tikzpicture}
    [scale=.4,auto=left,node distance=1.5cm, every node/.style={circle,draw}]
  
  \node  (n2) at (5,2) [fill=white]  {\small{2}};
  \node  (n5) at (13,2) [fill=white]  {\small{5}};
  \node  (n1) at (21,2) [fill=white]  {\small{1}};
  \node  (n4) at (5,-3) [fill=white]  {\small{4}};
  \node  (n6) at (13,-3) [fill=white]  {\small{6}};
  \node  (n3) at (21,-3) [fill=white]  {\small{3}};
\draw[<->, thick] (n2) edge (n4); 
\draw[<->, thick] (n2) edge (n5); 
\draw[<->, thick] (n2) edge (n6); 
\draw[<->, thick] (n4) edge (n5); 
\draw[<->, thick] (n4) edge (n6); 
\draw[<->, thick] (n1) edge (n3); 
\draw[<->, thick] (n1) edge (n5); 
\draw[<->, thick] (n1) edge (n6); 
\draw[<->, thick] (n3) edge (n5); 
\draw[<->, thick] (n3) edge (n6); 
\draw[<->, thick] (n5) edge (n6); 
\end{tikzpicture} 
}
\end{center}
\caption{A $6$-vertex symmetric directed graph $G$ admitting quotient graphs that are symmetric and non symmetric, as for example,  the quotient graph $Q_1 = G / \bowtie_1$ in Figure~\ref{fig:exq1} (a)  is symmetric and the quotient graph $Q_2 = G/ \bowtie_2$ in Figure~\ref{fig:exq1} (b)  is non symmetric.}
\label{fig:exf}
\end{figure}
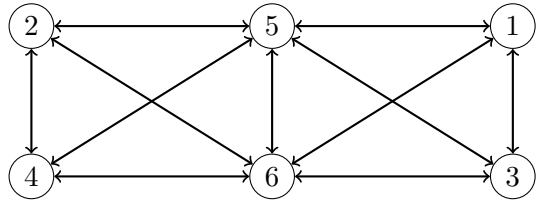

\begin{figure}[H] 
\begin{center}
\begin{tabular}{cc} 
\tiny{ 
\begin{tikzpicture}
    [scale=.4,auto=left,node distance=1.5cm, every node/.style={circle,draw}]
  
  \node  (n5) at (13,2) [fill=white]  {\small{5}} edge [loop left, thick] ();
  \node  (n1) at (21,2) [fill=white]  {\small{1}};
  \node  (n3) at (21,-3) [fill=white]  {\small{3}};
\draw[<->, thick] (n1) edge (n3); 
\draw[<->, thick] (n1) edge  [bend left=-10] (n5); 
\draw[<->, thick] (n1) edge [bend left=10] (n5); 
\draw[<->, thick] (n3) edge [bend left=-10] (n5); 
\draw[<->, thick] (n3) edge [bend left=10] (n5); 
\end{tikzpicture} 
} \qquad \qquad & \qquad  \qquad 
\tiny{ 
\vspace{2mm} 
\begin{tikzpicture}
    [scale=.4,auto=left,node distance=1.5cm, every node/.style={circle,draw}]  
  \node  (n5) at (13,2) [fill=white]  {\small{5}};
  \node  (n1) at (21,2) [fill=white]  {\small{1}} edge [loop right, thick] ();
  \node  (n6) at (13,-3) [fill=white]  {\small{6}};
\draw[<->, thick] (n1) edge [bend left=-15] (n5); 
\draw[->, thick] (n1) edge [bend left=-75] (n5); 
\draw[->, thick] (n1) edge[bend left=-30]  (n5); 
\draw[->, thick] (n1) edge[bend left=-45]  (n5); 
\draw[->, thick] (n1) edge [bend left=-60] (n5); 
\draw[<->, thick] (n1) [bend left=15] edge (n6); 
\draw[->, thick] (n1) edge [bend left=30] (n6); 
\draw[->, thick] (n1) edge[bend left=45]  (n6); 
\draw[->, thick] (n1) edge[bend left=60]  (n6); 
\draw[->, thick] (n1) edge [bend left=75] (n6); 
\draw[<->, thick] (n5) edge (n6); 
\end{tikzpicture} 
}
\end{tabular}
\end{center}
\caption{(a) A symmetric graph $Q_1$ which is the quotient $G / \bowtie_1$ of the symmetric directed graph $G$ in Figure~\ref{fig:exf} by the balanced equivalence relation $\bowtie_1$ with classes $\{1,2\},\, \{3,4\}$ and $\{5,6\}$. (b) A non symmetric graph $Q_2$ which is the quotient  $G / \bowtie_2$ by the balanced equivalence relation with classes $\{1,2,3,4\},\, \{5\}$ and $\{6\}$.}
\label{fig:exq1}
\end{figure}
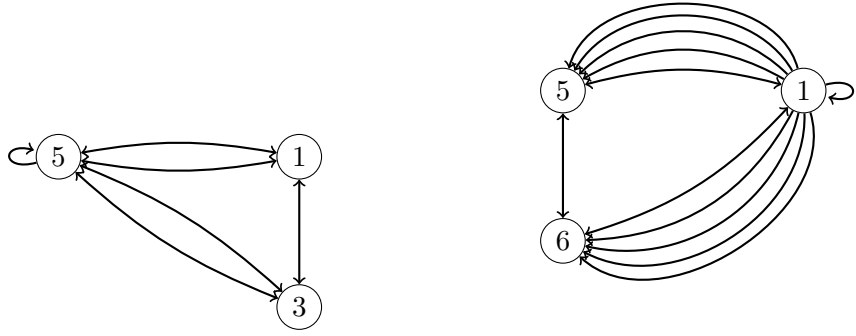

\begin{corollary}\label{cor:connectedquo}
If an $n$-vertex symmetric directed graph has an $m$-vertex connected symmetric quotient graph for a balanced relation of its set of vertices, where $n >m$, then $n$ is a multiple of $m$. 
\end{corollary}

We apply now the previous results to characterize the lifts of symmetric  directed graphs using balanced equivalence relations. 

\begin{corollary} \label{cor:consequences} Let $Q$ be a symmetric directed graph. We have that for any symmetric lift $G$ of $Q$, where $Q = G/ \bowtie$ and $\bowtie$ is a balanced equivalence relation at the set of vertices of $G$, any two connected vertices in $Q$ must unfold to the same number $r$ of vertices in $G$. Moreover, the lift graph $G$ has no multiple connections if and only if that number of vertices equals the number of connections between the two quotient vertices. 
\end{corollary}

\begin{theorem} \label{thm:existence of lifts}
Let $Q$ be an $m$-vertex symmetric directed connected graph with multiple edges. Take $p$ to be the maximum of the entries of the adjacency matrix of $Q$. 
Then:\\
(i) The symmetric lifts $G$ of $Q$, where $Q = G/ \bowtie$ for balanced $\bowtie$, with no multiple edges have $mr$ vertices with $r \geq p$. \\
(ii) For any $r \geq p$ there are symmetric lifts $G$ of $Q$, where  $Q = G/ \bowtie$ for balanced $\bowtie$,  with no multiple edges and  $mr$ vertices.  
\end{theorem}

\begin{proof} Let $Q$ be an $m$-vertex symmetric directed connected graph with multiple edges and take $p$ to be the maximum of the entries of the adjacency matrix of $Q$. Let $G$ be a symmetric lift of $Q$, where $Q = G/ \bowtie$ for balanced $\bowtie$.  \\
(i) As $Q$ is symmetric, applying Corollary~\ref{cor:consequences}, it follows that if there is any connection between two vertices then they both unfold to the same number $r$ of vertices in $G$. As $Q$ is connected, it follows that there is a directed path linking all the vertices of $Q$. Thus all the vertices must unfold  to the same number $r$ of vertices in $G$. As it is required that $G$ has no multiple edges, then $r \geq p$. Thus any such lift has $mr$ vertices with $r \geq p$. \\
(ii) Let $r \geq p$ and denote the  adjacency matrix of $Q$ by $A_{Q} = [q_{ij}]_{m \times m}$. Following the procedure in the second part of the proof of Theorem~\ref{thm:quobidi} , with $k_1=\cdots=k_m = r$, we build, from the matrix $A_{Q}$,  a symmetric matrix $A$ of order $mr$ with  all the entries equal to 0 or $1$. By construction, the symmetric matrix $A$ is the adjacency matrix of a symmetric lift of $Q$ with no multiple edges,  $mr$ vertices where each vertex in $Q$ unfolds to $r$ vertices, and where $Q = G/ \bowtie$ for balanced $\bowtie$.  
\end{proof}

\begin{example}\normalfont  \label{example:nonuniquiness}
Consider the $3$-vertex symmetric directed graph presented in Figure~\ref{fig:quosym}. Note that it is connected and it has adjacency matrix 
\[
 \left( 
\begin{array}{lll}
0 & 1 & 3 \\
1 & 0 & 2\\
3 & 2 & 0
\end{array}\right)\, .
\]
Two possible symmetric lifts with no multiple edges are the $9$-vertex and $12$-vertex networks with adjacency matrices given, respectively, by 
\[
{\small 
\left( 
\begin{array}{lll|lll|lll}
0 & 0 & 0 & 1 & 0 & 0 & 1 & 1 & 1 \\
0 & 0 & 0 & 0 & 1 & 0 & 1 & 1 & 1 \\
0 & 0 & 0 & 0 & 0 & 1 & 1 & 1 & 1 \\
\hline 
1 & 0 & 0 & 0 & 0 & 0 & 1 & 1 & 0 \\
0 & 1 & 0 & 0 & 0 & 0 & 0 & 1 & 1 \\
0 & 0 & 1 & 0 & 0 & 0 & 1 & 0 & 1 \\
\hline 
1 & 1 & 1 & 1 & 0 & 1 & 0 & 0 & 0 \\
1 & 1 & 1 & 1 & 1 & 0 & 0 & 0 & 0 \\
1 & 1 & 1 & 0 & 1 & 1 & 0 & 0 & 0 
\end{array}\right), \ \ \ 
\left( 
\begin{array}{llll|llll|llll}
0 & 0 & 0 & 0    & 1 & 0 & 0 & 0      & 1 & 1 & 1 & 0 \\
0 & 0 & 0 & 0    & 0 & 1 & 0 & 0      & 0 & 1 & 1 & 1\\
0 & 0 & 0 & 0    & 0 & 0 & 1 & 0      & 1 & 0 & 1 & 1 \\
0 & 0 & 0 & 0    & 0 &0 & 0 & 1      & 1 & 1 & 0 & 1 \\
\hline 
1 & 0 & 0 & 0    & 0 & 0 & 0 & 0      & 1 & 1 & 0 & 0 \\
0 & 1 & 0 & 0    & 0 & 0 & 0 & 0      & 0 & 1 & 1 & 0 \\
0 & 0 & 1 & 0    & 0 & 0 & 0 & 0      & 0  & 0 & 1 & 1 \\
0 & 0 & 0 & 1    & 0 & 0 & 0 & 0      & 1  & 0 & 0 & 1 \\
\hline  
1 & 0 & 1 & 1    & 1 & 0 & 0 & 1      & 0 & 0 & 0 & 0 \\
1 & 1 & 0 & 1    & 1 & 1 & 0 & 0      & 0 & 0 & 0 & 0 \\
1 & 1 & 1 & 0    & 0 & 1 & 1 & 0      & 0 & 0 & 0 & 0 \\
0 & 1 & 1 & 1    & 0 & 0 & 1 & 1      & 0 & 0 & 0 & 0 \\
\end{array}\right)\, .}
\]
\END
\end{example}

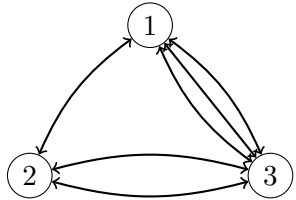
\begin{figure}[H] 
\begin{center}
\tiny{ 
\begin{tikzpicture}
    [scale=.4,auto=left,node distance=1.5cm, every node/.style={circle,draw}] 
  \node  (n1) at (9,2) [fill=white]  {\small{1}};
  \node  (n2) at (5,-3) [fill=white]  {\small{2}};
  \node  (n3) at (13,-3) [fill=white]  {\small{3}};
 \draw[<->, thick] (n1) edge[bend left=-15] (n2); 
  \draw[<->, thick] (n1) edge[bend left=-15] (n3);
  \draw[<->, thick] (n1) edge[bend left=15] (n3);
   \draw[<->, thick] (n1) edge[bend left=0] (n3);
\draw[<->, thick] (n2) edge[bend left=-15] (n3);
\draw[<->, thick] (n2) edge[bend left=15] (n3);
\end{tikzpicture} 
}
\end{center}
\caption{A $3$-vertex symmetric directed graph $Q$ admitting symmetric lifts $G$, namely, $Q = G/ \bowtie$ for balanced $\bowtie$ where $G$ is a symmetric graph.}
\label{fig:quosym}
\end{figure}

\section{Application to coupled cell systems} \label{sec:ApCCN}

We apply the results obtained in Section~\ref{sec:QLSG} to the study of two classes of coupled dynamical systems. We follow the formalism of coupled cell networks of Golubitsky, Stewart and co-workers~\cite{SGP03,GST05, GS06}, where the cells (the dynamical systems) and the couplings between the cells are abstracted through a graph -- the coupled cell network. The classes we consider are the gradient and Hamiltonian  coupled cell systems 
(Sections~\ref{subsec:gradient} and \ref{subsec:Hamiltonian}, respectively).

The coupled cell systems {\it admissible} by a given graph $G$ must be consistent with $G$ in the following way. A set of ordinary differential equations is associated with each vertex (cell). Thus  a  finite-dimensional vector space $P_v$ is assigned to each vertex $v \in V$, where
$V$ denotes the set  of vertices of the graph. The total configuration space is then $P =  \prod_{v \in V} P_v$. Let $x_v$ denote coordinates on $P_v$. The space $P_v$ is called the {\em cell phase space} of $v$ and $P$ is the {\em total phase space} of the coupled cell system consistent with $G$. Any coupled cell system associated with $G$, for a given choice of $P_v$, is 
\[
{\dot{{ x}}} = F({{x}}), \qquad (x \in P)
\]
where the system associated with vertex $v$ takes the  form 
\[
\dot{x}_v = g_v \left( x_v, x_{v_1}, \ldots, x_{v_\ell}\right).
\]
The first argument $x_v$ in $g_v$ represents the internal dynamics of the cell $v$ and each of the remaining variables $x_{v_p}$ indicates an input edge from cell $v_p$ to the cell $v$. As we are assuming that the input edges  directed to any cell $v$ are of the same type,  we have that $g_v$ is invariant under any permutation of the corresponding input variables. Moreover, systems associated with cells of the same valency are governed by the same function $g_v$. Vector fields $F$ satisfying the above properties are called $G$-{\it admissible}.  Thus if $G$ is a  regular graph with valency $\ell$, then the coupled cell systems are determined by a unique $g_v$, say $g$, and have the form
\[
\dot{x}_v = g\left( x_v; \overline{x_{v_1}, \ldots, x_{v_\ell}}\right),
\]
where the overbar  indicates the invariance of $g$ under any permutation of the vertex coordinates $x_{v_1}, \ldots, x_{v_\ell}$.

\begin{example} \normalfont 
Consider the Petersen graph in Figure~\ref{fig:Peter} which is regular of valency $3$. If each vertex phase space is $\R^m$, for some $m \geq 1$,   the total phase space is $P = (\R^{m})^{10}$ and the associated coupled cell systems have the form 
\begin{equation}
\begin{array}{l}
\dot x_{1} =g(x_{1},\overline{x_{2},x_{5}, x_{6}}) \\
\dot x_{2} =g(x_{2},\overline{x_{1},x_{3}, x_{7}}) \\
\dot x_{3} =g(x_{3},\overline{x_{2},x_{4}, x_{8}}) \\
\dot x_{4} =g(x_{4},\overline{x_{3},x_{5}, x_{9}}) \\
\dot x_{5} =g(x_{5},\overline{x_{1},x_{4}, x_{10}}) \\
\dot x_{6} =g(x_{6},\overline{x_{1},x_{8}, x_{9}}) \\
\dot x_{7} =g(x_{7},\overline{x_{2},x_{9}, x_{10}}) \\
\dot x_{8} =g(x_{8},\overline{x_{3},x_{6}, x_{10}}) \\
\dot x_{9} =g(x_{9},\overline{x_{4},x_{6}, x_{7}}) \\
\dot x_{10} =g(x_{10},\overline{x_{5},x_{7}, x_{8}}) 
\end{array},
\end{equation}\label{eq:example}
for some smooth function $g:\ \left({\bf R}^m\right)^4 \to {\bf R}^m$.  
\END
\end{example}

\vspace{1cm}

Synchrony, one of the phenomena explored by the coupled cell networks formalism, is an important concept from the point of view of the dynamics of the associated coupled cell systems. Take an equivalence relation $\bowtie$ on the vertex set of a graph and consider the associated polydiagonal subspace $\Delta_{\bowtie}$. In this setup, a {\em synchrony subspace} $\Delta_{\bowtie}$ of $G$ is a polydiagonal subspace which is flow-invariant under any coupled cell system associated with the graph structure of $G$. Moreover, a crucial result that caracterizes the relations $\bowtie$ for which  $\Delta_{\bowtie}$ is indeed a synchrony subspace uses precisely the balanced property. Theorem 4.3 of \cite{GST05}, followed by Corollary~2.10 and Theorem 4.2 of  \cite{AD14}, can be stated in the following way: 

\begin{theorem} \label{thm:main}
Given  a graph $G$, we have that the following statements are equivalent:\\
(i) An equivalence relation $\bowtie$ on the graph set of vertices is balanced. \\
(ii) For any choice $P$ of the total phase space, the polydiagonal subspace $\Delta_{\bowtie}$ is a graph synchrony subspace. \\
(iii) The polydiagonal  subspace $\Delta_{\bowtie}$  is flow-invariant for {\em all linear} admissible vector fields taking the vertex phase spaces to be $\R$.\\
(iv) The graph adjacency matrix leaves the polydiagonal subspace $\Delta_{\bowtie}$ invariant. 
\end{theorem} 

Hence, we have two types of criteria to enumerate the synchrony subspaces of a graph: one is algebraic using the graph adjacency matrix; the other is graph theoretical, by finding the balanced equivalence relations (colorings) of the set of vertices of the graph.

Fixing a graph $G$ and one of its synchrony subspaces $\Delta_{\bowtie}$, by Theorem 5.2 of \cite{GST05} we have that the restriction of any coupled cell system associated with $G$ is a coupled cell system consistent with the quotient graph $Q = G/\bowtie$. Moreover, any admissible vector field of $Q$ can be seen as the restriction to $\Delta_{\bowtie}$  of an admissible vector field of $G$. 

\begin{example}\normalfont 
Consider the Petersen graph $G$ (Figure~\ref{fig:Peter}) and the balanced equivalence relation $\bowtie_2 = \left\{ \{1,3,9,10\},\, \{ 2,7\},\, \{4,5, 6, 8\}\right\},$ corresponding to the coloring of its  vertices  presented in Figure~\ref{fig:3colorPeter}. Choosing the representatives  $1,2,4$ of the $\bowtie_2$-classes, we have that  the equations (\ref{eq:example}) restricted to $\Delta_{\bowtie_2}$ are 
\[
\begin{array}{l}
\dot x_{1} =f(x_{1},\overline{x_{4},x_{2}, x_{4}}) \\
\dot x_{2} =f(x_{2},\overline{x_{1},x_{1}, x_{2}}) \\
\dot x_{4} =f(x_{4},\overline{x_{4},x_{1}, x_{1}}) 
\end{array}, 
\]
which are consistent with the quotient graph $Q_2 = G / \bowtie_2$ (Figure~\ref{fig:3colorquo}). 
\END
\end{example}

The importance of the existence of flow-invariant subspaces under the dynamics of a system has been long recognized. For example,  flow-invariant subspaces can imply the existence of non-generic dynamical behavior such as heteroclinic cycles and networks, leading to complex  dynamics.  In the context of coupled cell systems, synchrony subspaces are examples of flow-invariant spaces, having an important  role  in understanding the dynamics of the coupled cell systems associated with a given graph. 
Moreover,  the restriction of the dynamics of a coupled cell system to a synchrony subspace is again a coupled cell system whose structure is consistent with a graph that has fewer vertices (the quotient graph), with  a lower-dimensional phase space, and potentially with known dynamics. The dynamics of the quotient can then be lifted to the overall phase space and give full information concerning the dynamics on those synchrony subspaces. Moreover,  the knowledge of the set of all synchrony subspaces of a coupled cell graph and the associated quotients can be used, for example, to investigate the possibility of the associated coupled cell systems to support  heteroclinic behavior. See for example~\cite{AADF11, AD17}.

In \cite{MR14}, gradient coupled cell systems are addressed and it  is shown that only symmetric directed graphs can have admissible coupled cell systems with gradient structure. An analogous result is obtained in~\cite{CBP17} for Hamiltonian coupled cell systems. Moreover, there are in general coupled cell systems associated with symmetric directed graphs that are neither  gradient nor Hamiltonian. Nevertheless, for symmetric graphs, there are always coupled cell systems with each one of those  structures, as explained in \cite{MR14, CBP17}.  The next two sections are devoted  to these two types of coupled cell systems. Using the results of Section~\ref{sec:QLSG}, we show that the extra structure of being gradient or Hamiltonian for the coupled cell systems admissible by a (symmetric) network $G$ can be lost by the systems admissible by a (non-symmetric) quotient network $Q$ of $G$. Conversely, we see that gradient (Hamiltonian) dynamics can appear for a symmetric quotient network $Q$ of a symmetric or a non-symmetric network $G$ but whose associated dynamics is not gradient (Hamiltonian). 

Throughout,  the total number of cells  in a  coupled system is denoted by $n$  and we take coordinates on the (identical) cells so that the cells state space is $\r^m$, for some $m \geq1$.

\subsection{Gradient coupled cell systems} \label{subsec:gradient}

In this section we investigate quotients and lifts of symmetric graphs related to gradient structure,  that is, associated with a coupled dynamical system given by the  (negative) gradient  of a smooth function $f: P \to \r$,
\begin{equation} \label{eq:gradient}
\dot{x} \ = \ - \nabla f(x), 
\end{equation}
where $x = (x_1, \ldots, x_n)$ denotes the state variable of the total phase space  $P = (\r^m)^n$ and, for $i=1, \ldots n,$  $x_i \in \r^m$ is the variable of cell $i$.   

The function $f$ whose gradient is an admissible vector field for a given graph $G$ is called an {\it admissible gradient function} for $G$. Hence, 
the gradient structure of a coupled system is in one-to-one correspondence with the existence of an admissible gradient function associated with it. We also notice that a necessary condition for an admissible vector field to be gradient  is that  each individual cell to be also gradient. 

In \cite{MR14} the characterization of admissible gradient functions  is given for symmetric graphs with no multiple edges. We remark that the result generalizes directly  for symmetric graphs with multiple edges and loops. The result is given in Theorem~\ref{thm:AGF} below. 

We first recall that a graph is called {\em bipartite} if its set of vertices $\V$ can be devided into two disjoint subsets $\V_1$ and $\V_2$ such that every edge of $G$ connects a vertex in a subset to a vertex in the other. 
If the graph is bipartite, we reorder the numbering of cells if necessary so that $\V_1 = \{1, \ldots, r\}$ and $\V_2 = \{r+1, \ldots, n\}$, $r >1$, and $A_G$ is an order-$n$   matrix of the block form
\[ A_G = \left(    \begin{array}{cc}
0_r & B \\
C & 0_{n-r}  \end{array}  \right). \]
 We then have:  

\begin{theorem} \label{thm:AGF}
Let $G$ be a symmetric graph that may have multiple edges and loops, and let $A_{G}~ =~(a_{ij})$ denote the adjacency matrix of the graph $G$. A  function  $f: P \to { \r}$ is an  admissible gradient function associated with $G$ if, and only if, 
there exist  smooth functions $\alpha : \r^m \to \r$ and
$\beta : \r^{2m}  \to \r$ such that 
\begin{equation} \label{eq:AGF}
 f(x) \ = \   \sum_{i=1, \ i \leq j}^n   a_{ij} \beta(x_i, x_j) + \sum_{i=1}^{n} \alpha(x_{i}).
\end{equation}
If $G$ is bipartite and in addition there exist $v_i \in \V_1$ and $v_j \in \V_2$ with same valency, or if $G$ is not bipartite, then $\beta$ is invariant under permutation of the vector variables. 
\end{theorem}

\begin{proof}
It follows  the steps of the proof of \cite[Theorem 2.4]{MR14}, just observing that here the existence of an edge between two vertices in the graph may not be unique. Also, for the proof of the non-bipartite case, instead of taking the maximal subgraph containing no odd cycles we take the maximal subgraph containing no odd cycles nor multiple edges.  
\end{proof}

\begin{example} \label{ex:S2} \normalfont 
The simplest graph with multiple edges is given in Figure~\ref{fig:quoD4}. From Theorem~\ref{thm:AGF}, choosing the cell phase spaces to be $\r$,  the function $f: \r^2 \to \r$,
\begin{equation} \label{eq:AGFquoD4}
 f(x_1, x_2) =  - 2x_1^2 x_2 - 2 x_1 x_2^2 
\end{equation} 
is an admissible gradient function for this graph.
\END
\end{example}

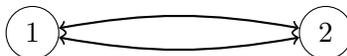
\begin{figure}[H] 
\begin{center}
\begin{tikzpicture}
 [scale=.15,auto=left, node distance=1.5cm, every node/.style={circle,draw}]
 \node[fill=white] (n4) at (4,2) {\small{1}};
 \node[fill=white] (n1) at (30,2)  {\small{2}};
 \draw[<->, thick] (n1) edge  [bend left=-10] (n4); 
\draw[<->, thick] (n1) edge [bend left=10] (n4); 
\end{tikzpicture}
\end{center}
\caption{The simplest 2-cell graph with multiple edges.}
\label{fig:quoD4}
\end{figure}

\begin{example} \normalfont
Consider the graph of Figure~\ref{fig:exq1}.  From Theorem~\ref{thm:AGF}, the general form of the admissible gradient functions associated with this graph is given by
\begin{equation} \label{eq:exq1}
 f(x_1, x_3, x_5) = \beta(x_1, x_3) + 2 \beta(x_1, x_5)  + 2 \beta(x_3, x_5) + \beta(x_5, x_5) +  \alpha(x_1) + \alpha(x_3) +\alpha(x_5), 
\end{equation}
with  $\beta(x,y) = \beta(y,x)$. 
\END
\end{example}

It is a consequence of Theorem~\ref{thm:existence of lifts} that any symmetric directed graph admits a symmetric directed lift. We remark that an admissible vector field associated with a symmetric graph  can admit a restriction to an invariant subspace with a gradient dynamics even if it is not of a gradient type, as we illustrate in the next example.

\begin{example} \normalfont 
Consider  the symmetric directed graph of Figure~\ref{fig:D4}. We recall from \cite[Lemma 2.2]{MR14} that for any graph with more than three vertices, a necessary condition for a function $f$ to be admissible for this graph is that
\begin{equation} \label{eq:3rdderivative}
 \frac{\partial^3 f}{\partial x_{i} \partial x_{j} \partial x_{k}} \equiv 0,  
 \end{equation}
for any two-by-two distinct cells $i, j, k$. It then follows that the system of equations
\begin{equation} 
\begin{array}{l}
\dot x_{1} = x_2 x_4 + x_1 (x_2+x_4) \\
\dot x_{2} =x_1x_3 + x_2(x_1+x_3) \\
\dot x_{3} = x_2 x_4 + x_3 (x_2+x_4)  \\
\dot x_{4} = x_1 x_3 + x_4 (x_1+x_3)  \\
\end{array} \label{eq:altereiD41}
\end{equation}
defines a coupled cell system admissible for the graph of Figure~\ref{fig:D4} which is not gradient. As the equivalence relation $\bowtie$ with classes $\{1,3\}$ and $\{2,4\}$ is balanced for this graph, we have that $\Delta_{\bowtie} = \{ {\bf x}:\, x_1 = x_3, \ x_2 = x_4\}$ is a synchrony subspace. Moreover, the corresponding quotient is the graph given in Figure~\ref{fig:quoD4}.  Equations~(\ref{eq:altereiD41}) restricted to  $\Delta_{\bowtie}$ correspond to:
\begin{equation} \label{eq:quoD4}
\begin{array}{l}
\dot x_{1} = x_2^2 + 2 x_1 x_2 \\
\dot x_{2} =x_1^2 + 2 x_2x_1,
\end{array} \, .
\end{equation}
corresponding to a gradient coupled cell system admissible for the graph in Figure~\ref{fig:quoD4}. Note that $f/2$, for $f$ as in (\ref{eq:AGFquoD4}), is an admissible gradient function for it.
\END
\end{example}

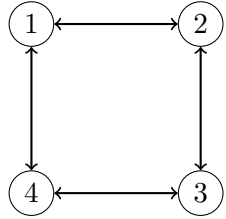
\begin{figure}[H] 
\begin{center}
\tiny{
\begin{tikzpicture}
 [scale=.15,auto=left, node distance=1.5cm, every node/.style={circle,draw}]
 \node[fill=white] (n4) at (1,1) {\small{4}};
 \node[fill=white] (n1) at (1,16)  {\small{1}};
 \node[fill=white] (n2) at (16,16)  {\small{2}};
\node[fill=white] (n3) at (16,1)  {\small{3}};
 \foreach \from/\to in {n4/n3,n4/n1,n2/n1,n2/n3}  \draw[<->, thick] (\from) -- (\to);
\end{tikzpicture}
\caption{The symmetric regular $4$-cell graph  of valency $2$.} 
\label{fig:D4}}
\end{center}
\end{figure}

We finally observe that, as more naturally expected,  gradient admissible vector fields can also be lifted to  vector fields that are still gradient.

\begin{example} \normalfont  
The vector field (\ref{eq:quoD4}) associated with the graph of Figure~\ref{fig:quoD4} can be lifted to an admissible  vector field for  the graph of Figure~\ref{fig:D4}  which is gradient. In fact,  taking $\tilde{f}: \r^4 \to \r$, 
\[\tilde{f}(x_1, x_2, x_3, x_4) =  \beta(x_2, x_1) + \beta(x_2, x_3) +  \beta(x_4, x_1) + \beta(x_4, x_3),  \]
for $\beta(x,y) =- 1/2( x^2y + x y^2)$, we have that $\tilde{f}$ is an admissible gradient function for the graph of Figure~\ref{fig:D4} which restricts to (\ref{eq:AGFquoD4}) on the synchrony subspace $\Delta$.  That is, the restriction of the admissible gradient function $\tilde{f}$ for the graph of Figure~\ref{fig:D4} is an admissible gradient function for its quotient graph in  Figure~\ref{fig:quoD4}.
\END
\end{example}

\subsection{Hamiltonian coupled cell systems} \label{subsec:Hamiltonian}

In this section quotients and lifts of symmetric graphs related to a Hamiltonian dynamical systems are investigated. A Hamiltonian coupled dynamical system is a system defined from ordinary differential equations  
\begin{equation} \label{eq:Hamiltonian}
(\dot{q}, \dot{p})^t \ = \ J \nabla h(q,p), 
\end{equation}
where 
\[  J \ = \ \left( \begin{array}{cc}
0_{nm}  & I_{nm} \\
-I_{nm} & 0_{nm} 
\end{array} \right).  \]
The variable vectors $q = (q_1, \ldots, q_n)$  and $p = (p_1, \ldots, p_n)$ are called position and momentum vectors, respectively, and form the  total phase space  $P = (\r^{2m})^n$ where, for $i=1, \ldots n,$  $(q_i, p_i) \in \r^{2m}$ are the position and momentum variables of each cell.   

  The function $h: \r^{mn} \times \r^{mn} \to \r$ satisfying (\ref{eq:Hamiltonian})  is the Hamiltonian function of the system, and when this system of equations is a coupled system associated with a given graph $G$ this is called an {\it admissible Hamiltonian function} for $G$. Hence, 
the Hamiltonian structure of a coupled system is in one-to-one correspondence with the existence of an admissible Hamiltonian function associated with it. In addition, a necessary condition for an admissible vector field to be Hamiltonian is that  each individual cell be also Hamiltonian. 

In the following theorem we characterize the admissible Hamiltonian functions of any symmetric graph. This corresponds  to Theorem~\ref{thm:AGF}  of the previous section adapted to the Hamiltonian setup.

\begin{theorem} \label{thm:AHF}
Let $G$ be a symmetric graph that may have multiple edges and loops, and let $A_{G}~ =~(a_{ij})$ denote the adjacency matrix of the graph $G$. A  function  $h: P \to { \r}$ is an  admissible Hamiltonian function associated with $G$ if, and only if, 
there exist  smooth functions $\alpha : \r^{2m} \to \r$ and
$\beta : \r^{2m} \times \r^{2m} \to \r$ such that 
\begin{equation} \label{eq:AHF}
 h(q,p) \ = \   \sum_{i=1, \ i \leq j}^n   a_{ij} \beta(q_i, q_j, p_i, p_ j) + \sum_{i=1}^{n} \alpha(q_{i}, p_i).
\end{equation}
If $G$ is bipartite and in addition there exist $v_i \in \V_1$ and $v_j \in \V_2$ with same valency, or if $G$ is not bipartite, then $\beta$ is invariant under the permutation 
$\sigma(q_i, q_j, p_i, p_j) = (q_j, q_i, p_j, p_i) $. 
\end{theorem}

\begin{proof}
It is the same as the proof of Theorem~\ref{thm:AGF}. To see this,  consider the similarity  between  the expressions  (\ref{eq:AGF}) and  (\ref{eq:AHF}) so that  the algebraic arguments used to prove the result for an admissible gradient function in the first case hold to construct an  admissible Hamiltonian function for the second. For the present case, it is crucial to notice that   (\ref{eq:3rdderivative}) holds by replacing each $x_l$, for
$l = i,j,k,$  by  $q_l$ and $p_l$. 
\end{proof}

\begin{example} \normalfont
Consider the graph of Figure~\ref{fig:quoD4}. Taking the cell phase spaces to be $\r$, the function $h: \r^4 \to \r$ defined by
\begin{equation} \label{eq:AHFquoD4}
 h(q_1, q_2, p_1, p_2)  = 2 p_1^2 q_2 + 2 p_2^2 q_1, 
 \end{equation}
is an admissible Hamiltonian function for this graph.
\END 
\end{example}

As expected from the discussion in Section~\ref{subsec:gradient} for gradient coupled systems, an admissible vector field associated with a symmetric graph  can admit a restriction to an invariant subspace with a Hamiltonian dynamics even if it is not Hamiltonian, as we illustrate in the next example.

\begin{example} \normalfont 
 As before, consider the symmetric directed graph of Figure~\ref{fig:D4}. Note that the following coupled cell system
\begin{equation} \label{eq:D41}
\begin{array}{l}
\dot q_{1} =  p_1 (q_2+q_4) \\
\dot q_{2} =p_2(q_1+q_3) \\
\dot q_{3} = p_3 (q_2+q_4)  \\
\dot q_{4} = p_4 (q_1+q_3)  \\
\dot p_{1} = - p_2 p_4  \\
\dot p_{2} =- p_1 p_3 \\
\dot p_{3} = - p_2 p_4   \\
\dot p_{4} = - p_1 p_3  \\
\end{array} 
\end{equation}
is admissible for that graph but it is not  Hamiltonian. However, its restriction to the synchrony subspace $\Omega = \{ ({\bf q},{\bf p}):\, (q_1, p_1) = (q_3,p_3) \ (q_2,p_2) = (q_4,p_4)\}$, leads to the  quotient graph given in Figure~\ref{fig:quoD4} with reduced equations
\begin{equation} \label{eq:HquoD4}
\begin{array}{l}
\dot q_{1} = 2 p_1 q_2 \\
\dot q_{2} = 2 p_2 q_1 \\
\dot p_{1} = - p_2^2  \\
\dot p_{2} = - p_1^2  \\
\end{array} \, .
\end{equation}
These equations define an admissible Hamiltonian  coupled cell system for the quotient graph in  Figure~\ref{fig:quoD4}. Note that $h/2$, for $h$ as in (\ref{eq:AHFquoD4}), is an admissible Hamiltonian function for this quotient graph. 
\END 
\end{example}

Again as for gradient coupled systems,  admissible Hamiltonian systems can also be lifted to  vector fields that are still Hamiltonian.

\begin{example} \normalfont 
The admissible and Hamiltonian coupled cell system  (\ref{eq:HquoD4}) associated with the graph of Figure~\ref{fig:quoD4} can be lifted to a coupled cell system which is admissible and Hamiltonian for  its lift of Figure~\ref{fig:D4}. Notice that $\tilde{h}: (\r^4)^2 \to \r$ defined by 
\[\tilde{h}(q, p) =  \beta(q_2, q_1, p_2, p_1) + \beta(q_2, q_3, p_3, p_3) +  \beta(q_4, q_1, p_4, p_1) + \beta(q_4, q_3, p_4, p_3),  \]
for $\beta(x,y,z,w) =  1/2 (z^2y + 2w^2x)$, is an admissible Hamiltonian function for this graph which restricts to (\ref{eq:AHFquoD4}) on the synchrony subspace $\Omega$.  
\END 
\end{example}

\subsection{Gradient and Hamiltonian coupled cell systems}

The general form of the gradient admissible functions and  Hamiltonian admissible functions  for $G$ and  $Q$ are related in the following way:

\begin{proposition}
Let $G$ be a symmetric graph and let $\bowtie$ be a balanced relation on the set of vertices of $G$ defining a symmetric connected quotient graph $Q = G/\bowtie$. Let $f^G$ ($h^G$) to be a gradient (Hamiltonian)  admissible function of a gradient (Hamiltonian) coupled cell system of $G$. Then, there is a gradient (Hamiltonian)  admissible function $f^Q$ ($h^Q$)  for the coupled cell system of $G$ restricted to $\Delta_{\bowtie}$, such that 
\[  
f^G|_{\Delta_{\bowtie}} \ \equiv  \ k \  f^{Q}, \ \  h^G|_{\Delta_{\bowtie}} \ \equiv  \ k \  h^{Q}, 
 \]
where $k$ is the cardinality of the $\bowtie$-classes.
\end{proposition}

\begin{proof}
It follows from Theorem~\ref{thm:bidquo} that the cardinality $k$ of every $\bowtie$-class is the same. The result then follows from Theorems~\ref{thm:AGF} and \ref{thm:AHF} and the definition of quotient graph. 
\end{proof}

\begin{corollary}\label{prop:admfunctions}
Let $G$ be a symmetric graph and let $\bowtie$ be a balanced relation on the set of vertices of $G$ defining a symmetric quotient graph $Q = G/\bowtie$. Then:\\
(i) The gradient (Hamiltonian) admissible functions for $G$ restricted to the synchrony subspace determined by $\bowtie$ are  gradient (Hamiltonian) admissible functions for $Q$. \\
(ii) A gradient (Hamiltonian) admissible function for $Q$ is a restriction of a gradient (Hamiltonian) admissible function for $G$. \\
\end{corollary}

As we have shown, from our results, it follows that gradient and Hamiltonian dynamics can appear in coupled cell systems that are neither gradient nor Hamiltonian. This already comes from the fact that non symmetric graphs can often admit quotient graphs by balanced relations that are symmetric. But this comes also from the more subtle situation in the level of symmetric graphs, for which gradient and Hamiltonian structure can occur on proper synchrony subspaces. 

For a given  gradient (or Hamiltonian) coupled cell system associated to a symmetric graph $G$, we are also able to identify the flow invariant subspaces of the total phase space where the restriction of the dynamics is still gradient (or Hamiltonian).
These are precisely the subspaces defined by each balanced relation $\bowtie$ for which $Q = G/\bowtie$ is symmetric. 

\begin{example} \normalfont 
Consider  a gradient coupled cell system associated with the symmetric graph $G$ in Figure~\ref{fig:exf}, part of the gradient dynamics of the system appears in the restriction to the synchrony subspace $\Delta_{\bowtie_1} = \{ x: \, x_{1}  = x_{2},\  x_{3}  = x_{4},\ x_{5}  = x_{6} \}$ but not in the restriction to the synchrony subspace $\Delta_{\bowtie_2} = \{ x: \, x_{1}  = x_{2} = x_{3}  = x_{4} \}$. The same goes for Hamiltonian dynamics. 
\END
\end{example}

%There is also the recent work in ~\cite{DP09,DP10} that makes use of bidirectional graphs to characterize and enumerate the periodic patterns of synchrony in $n$-dimensional Euclidean lattice networks  with nearest neighbour coupling architecture. Given such a lattice and a periodic synchrony pattern, the quotient of the lattice by the synchrony space associated with the pattern is a bidirectional graph where each vertex receives an even number of inputs. That is, the bidirectional graph has even {\em degree} or {\em valency}. The results in \cite{DP09} relate the existence of those periodic patterns with the synchrony patterns of finite bidirectional coupled cell networks of even degree, making use of a classical theorem of graph theory concerning the factorisation of even degree regular graphs. 

\subsection*{Acknowledgments}The first two authors acknowledge partial support by CMUP (UID/MAT/00144/2013), which is funded by FCT (Portugal) with national (MEC) and European structural funds through the programs FEDER, under the partnership agreement PT2020. MM was supported by CAPES under  CAPES/FCT grant 88887.125430/2016-00. She also thanks  University of Porto for their great hospitality during the visit when this work was carried out.


\begin{thebibliography}{99}

\bibitem{AADF11} M. Aguiar, P. Ashwin, A. Dias, and M. Field. Dynamics of coupled cell networks: synchrony, heteroclinic cycles and inflation. {\it J. Nonlinear Sci.} {\bf 21} (2)  (2011) 271--323.

\bibitem{AD14}  M.A.D. Aguiar and  A.P.S. Dias. The Lattice of Synchrony Subspaces of a Coupled Cell Network: Characterization and Computation Algorithm. {\em Journal of Nonlinear Science} (2014) DOI 10.1007/s00332-014-9209-6

\bibitem{AD17} M.A.D.Aguiar and A.P.S.Dias. Heteroclinic network dynamics on joining coupled cell networks. 
{\it Dynamical Systems An International Journal} {\bf 32} (1) (2017) 4--22,  DOI 10.1080/14689367.2016.1197889 

\bibitem{ADGL09} M.A.D. Aguiar, A.P.S. Dias., M. Golubitsky and M. Leite. Bifurcations from regular quotient networks: a first insight. {\it Physica D} {\bf 238} (2)  (2009) 137--155.

\bibitem{Bronski et al} J.C. Bronski, L. De Ville, M.J. Park.  Fully synchronous solutions and the synchronization phase transition for the finite-N Kuramoto model. {\it Chaos: an Int. J. Non. Sci.} {\bf 22}(3) (2012) 033133.


\bibitem{BCP15}P.L. Buono, B. Chan and A. Palacios. Dynamics and Bifurcations in a $D_n$-symmetric Hamiltonian Network. Application to Coupled Gyroscopes. {\it Phys. D} {\bf 290} (2015)  8--23.

\bibitem{CBP17} B.S. Chan, P.L. Buono and A. Palacios. Topology and Bifurcations in Hamiltonian Coupled Cell Systems. {\it Dynamical Systems: An International  Journal} {\bf 32} (1)  (2017) 23--45.

\bibitem{DP09} A.P.S. Dias and E.M. Pinho. Spatially Periodic Patterns of Synchrony in Lattice Networks. {\it SIAM Journal of Applied Dynamical Systems} {\bf 8} (2) (2009) 641--675.

%\bibitem{DP10}  A.P.S. Dias and E.M. Pinho. Enumerating periodic patterns of synchrony via finite bidirectional networks. {\it Proceedings: Mathematical, Physical and Engineering Sciences} {\bf 466} (2115) (2010) 891--910.

\bibitem{F04} M. J. Field. Combinatorial Dynamics. {\em Dynamical Systems} {\bf 19} (2004) 217--243. 

\bibitem{GS06} M. Golubitsky and I. Stewart. Nonlinear dynamics of networks: the groupoid formalism. {\it Bulletin of American Mathematical Society} {\bf 43} (2006) 305--364.

\bibitem{GST05} M. Golubitsky, I. Stewart and A. T\"or\"ok. Patterns of synchrony in coupled cell networks with multiple arrows. {\em SIAM Journal of  Applied Dynamical  Systems} {\bf 4} (1) (2005) 78--100.

\bibitem{MR14} M. Manoel and M.R. Roberts. Gradient systems on coupled cell networks. {\it Nonlinearity} {\bf 28} (10) (2015) 3487--3509.

\bibitem{Rink Sanders} B. Rink, J. Sanders.  Coupled cell networks and their hidden symmetries. {\it SIAM J. Math. Anal.} {\bf 46}(2)  (2014)  1577--1609.


\bibitem{S07} I. Stewart. The lattice of balanced equivalence relations of a coupled cell network. {\it Math. Proc. Cambridge Philos. Soc.}  {\bf 143} (1) (2007) 165--183.

\bibitem{SGP03} I. Stewart, M. Golubitsky and M. Pivato. Symmetry groupoids and patterns of synchrony in coupled cell networks. {\it SIAM Journal of  Applied Dynamical  Systems} {\bf 2}(4) (2003) 609--646.

\bibitem{T17}  D.S Tourigny. Networks of planar Hamiltonian systems. {\it Comm. in Non. Sci. and Num. Sim.} {\bf 53} (2017) 263--277.



\end{thebibliography}
\end{document}